\documentclass{amsart}
\usepackage[all]{xy}
 \RequirePackage{amsgen}\RequirePackage{color}
\RequirePackage{amstext} \RequirePackage{amsbsy}
\RequirePackage{amsopn} \RequirePackage{amsthm}
\RequirePackage{amssymb} \RequirePackage{epsfig}
\RequirePackage{amscd}

\usepackage[mathscr]{eucal}

\newtheorem{lemma}{Lemma}
\newtheorem{cor}{Corollary}
\newtheorem{prop}{Proposition}

\newtheorem{theorem}{Theorem}

\theoremstyle{definition}

\theoremstyle{remark}

\newcommand{\mr}[1]{\mathrm{#1}}

    \let\c\gamma  
  
  \let\l\lambda

\def\C{\mathbb C}

\def\E{\mathcal E}

\def\H{\mathfrak H}

\def\Gal{{\rm Gal}}
\def\Frob{{\rm Frob}}
\def\F{{\mathbb F}}
  
\def\Z{{\mathbb Z}}

\def\Q{{\mathbb Q}}

\def\G{\Gamma}

\def\Ind{\mathrm{Ind}}
\def\SL{\mathrm{SL}_2(\mathbb Z)}
\begin{document}

\title[On  $\ell$-adic representations for noncongruence cuspforms]
{On  $\ell$-adic representations for a space of noncongruence cuspforms}

\author{Jerome William Hoffman}
\address{Department of Mathematics \\Louisiana State University\\
Baton Rouge, LA 70803\\USA} \email{hoffman@math.lsu.edu}

\author{Ling Long}
\address{Department of Mathematics\\Iowa State University\\ Ames, IA 50011 \\USA}
\email{linglong@iastate.edu}

\author{Helena Verrill}
\address{Warwick Mathematics Institute, UK} \email{H.A.Verrill@warwick.ac.uk}

\thanks{The second author was supported in part
by the NSA grant \#H98230-08-1-0076. Part of the work was done during Long's visit to University of California at Santa Cruz.
This research was initiated during
an REU summer program at LSU, supported by the
National Science Foundation grant DMS-0353722 and a
Louisiana Board of Regents Enhancement grant,
LEQSF (2002-2004)-ENH-TR-17.
The last author was partially supported by grants
LEQSF (2004-2007)-RD-A-16 and NSF award DMS-0501318.}

\date{today}

\begin{abstract}
This paper is concerned with  a compatible family of 4-dimensional $\ell$-adic representations
$\rho_{\ell}$ of $G_\Q:=\Gal(\overline \Q/\Q)$ attached to the space of weight 3 cuspforms
$S_3 (\Gamma)$ on a noncongruence subgroup $\Gamma \subset \SL$. For this representation
we prove that:
\begin{itemize}
\item[1.] {It is automorphic: the $L$-function $L(s, \rho_{\ell}^{\vee})$ agrees with the L-function for an automorphic form for $\text{GL}_4(\mathbb A_{\Q})$,} where $\rho_{\ell}^{\vee}$ is the dual of $\rho_{\ell}$.
\item[2.] For each prime $p \ge 5$ there is a basis $h_p = \{ h_p ^+, h_p ^- \}$ of $S_3 (\Gamma )$ whose expansion coefficients
satisfy 3-term Atkin and Swinnerton-Dyer (ASD) relations, relative to the $q$-expansion
coefficients of a newform $f$ of level 432. The structure of this basis depends on the class of $p$ modulo 12.
\end{itemize}
The key point is that the representation $\rho_{\ell}$ admits a quaternion multiplication
structure in the sense of \cite{ALLL09}.
\end{abstract}

%\sjc{11F11, 11F33}
 \maketitle

\section{Introduction}
\label{S;intro}
\subsection{}
\label{SS:intro1} Recall that a subgroup
of finite index  $\Gamma \subset \SL$ is a congruence subgroup if $\Gamma \supset \Gamma (N)$
for some integer $N\ge 1$, where  $\Gamma (N) \subset \SL$ is the normal subgroup consisting of matrices
congruent to the identity modulo $N$; $\Gamma$ is a noncongruence subgroup if it is not a congruence
subgroup. There is a vast
theory of modular forms on congruence subgroups (general reference for facts
and notation: \cite{Shi71}, \cite{d-s-mfbook}).
By contrast, modular forms on noncongruence subgroups
are less well-understood, and they exhibit qualitatively different behavior.
It is well known
that $S_{k}(\Gamma _0 (N),\chi )$
has a basis
of Hecke eigenforms, which have $q$-expansions
\[
f(z)=\sum_{n\ge 1} a_n (f)q^n, \quad \mr{where\ \ }q = \exp(2\pi iz),
\]
with
$a_n$ satisfying the relations
\begin{equation}
\label{eqn:Heckerelation} a_{np} - a_pa_n + \chi(p)p^{k-1}a_{n/p} =
0,\quad a_n = a_n(f)
\end{equation}
for all positive integers $n$ and primes $p\nmid N$, taking $a_{n/p}=0$ if
$p\nmid n$. Moreover, $a_p$ is the trace of Frobenius
%$\rho _f(\mr{Frob}_p) $
for a two-dimensional $\lambda$-adic representation
$\rho _f$ of  $G_\Q:=\Gal(\overline \Q/\Q)$ (\cite{D,DS,L}).

\subsection{}
\label{SS:intro2}
If $\Gamma$ is a noncongruence subgroup, there is in general no Hecke eigenbasis
for $S_k (\Gamma)$, the space of weight $k$ cuspforms for $\G$,
but rather it is conjectured that, at least in certain circumstances, for almost all primes $p$ there is a
basis $\{ h_j = h_{p, j} \}$ such that the $q$-expansion coefficients satisfy
3-term Atkin-Swinnerton-Dyer (ASD)
congruences
in the general shape:
\begin{equation}
\label{eqn:ASDrelation} a_{np}(h_j) - \alpha _p (j) a_n (h_j) + \chi _j (p)  p^{k-1}a_{n/p}(h_j) \equiv
0 \mr{\ mod\ }(np)^{k-1}
\end{equation}
where $|\alpha _p (j) | \le  2 p^{(k-1)/2}$  and $\chi _j (p)$ is a root of unity. In
\cite{Sch85i,Sch85ii}, A. J. Scholl proved the existence of $(2d+1)$-term ASD
congruences ($d = \dim S_k (\Gamma )$), under some standard assumptions such as the modular curve being defined over $\Q$ with infinity as a $\Q$-rational point. In fact, for every prime $\ell$, Scholl proved the
existence of a $G_{\Q}$-representation $\rho _{\ell} = \rho _{\Gamma , k, \ell}$
acting  on an $\ell$-adic space $W_{\ell }(\Gamma)$ of dimension $2d$
analogous to Deligne's construction for congruence subgroups. He also constructed $2d$-dimensional
$\Q _p$-vector spaces,  $V_p (\Gamma )$, with an action of a Frobenius operator and containing the subspace $S_k (\Gamma )\otimes \Q _p$, which are
the analogs in crystalline cohomology of the $\ell$-adic space $W_{\ell}(\Gamma )$. Scholl achieved the $(2d+1)$-term ASD congruences via a comparison theorem.  He managed to refine this to obtain 3-term congruences when the characteristic polynomials of those $(2d+1)$-term recursions have $d$ distinct $p$-adic roots.
In special cases involving  extra symmetries, such as 4-dimensional Scholl representations satisfying Quaternion Multiplications (see \cite{ALLL09}), we can find in a systematic way a basis of the noncongruence modular forms
whose members satisfy 3-term congruences, see Section \ref{S:aswd}. 

In recent studies (\cite{lly05}, \cite{long061}, \cite{all05},
\cite{l5}, \cite{ALLL09},
to which we send the reader for more background and precise conjectures) the $\alpha _p (j)$, up to multiplying by roots of unity in clear patterns,
are the $p$th $q$-expansion coefficients of newforms $f_j$ on {\it congruence} subgroups. To be more precise, there is a quadratic field $K$ such that $T^2-a_p(j)T+\chi_j(p)p^{k-1}$ is the Hecke polynomial at some place of { $K$}
over $p$ for some  automorphic form for $\text{GL}_2$ over { $K$} (see \cite[Thm. 4.3.2]{ALLL09} for details).

\subsection{}
\label{SS:intro3}
In this paper, we have a particular noncongruence group $\Gamma$ and $k = 3$, $d = 2$.
Experimentally it was discovered that the degree 4 polynomials  of the geometric Frobenii under the corresponding Scholl representation (denoted by $\rho_{\ell,2}$ below) factor into quadratic
pieces with coefficients in $\Q (\sqrt{-3})$, $\Q (\sqrt{-2})$ or $\Q (\sqrt{-6})$,
respectively as $p$ is congruent to $1 \mr{\ mod\ } 3$,  $5 \mr{\ mod\ } 12$, or $11 \mr{\ mod\ } 12$,
and moreover, the linear terms in these polynomials matched the $p$th Fourier coefficients
of a newform $f$ of level 432, up to multiplication by a twelfth root of unity. This trichotomy is explained by the existence of an order-8 quaternion group acting on  $\rho_{\ell,2}$. This falls into the general framework of 4-dimensional Galois representations admitting quaternion multiplications that is studied in detail in \cite{ALLL09}, except that the action of the quaternion group is not defined over a quadratic or biquadratic field. To overcome the extra complication, we use an auxiliary 4-dimensional $G_{\Q}$ representation $\rho_{\ell,4}$, which is a 4-dimensional Galois representation admitting quaternion multiplication defined over a biquadratic field. The automorphy of $\rho_{\ell,4}$ is known due to \cite{ALLL09}. It requires additional work to link $\rho_{\ell,4}$ to the level 432 newform $f$. We use $\rho_f$ to denote the 2-dimensional Deligne representation of $G_\Q$ attached to the Hecke newform $f$. The relations between $\rho_{\ell,2}$, $\rho_{\ell,4}$, and $\rho_{f}$ are depicted by the following diagram. For the precise statements, see
Corollary \ref{C:modular2},
Theorems \ref{thm:mod-main-1},  \ref{thm:mod-main-2}, and \ref{thm:mod-main}.

$${\small\xymatrix{\rho_{\ell,2} \ar@{=}[d]\ar@{--}[rr]\ar@{--}[dr] & &\rho_{f}\ar@{--}[d] \ar@{--}[dl]\\
\Ind_{G_{\Q(\sqrt{-3})}}^{G_\Q} \sigma_{\l,2,-3}  \ar@{--}[d]& \rho_{\ell,4}=\Ind_{G_{\Q(\sqrt{-3})}}^{G_\Q} \sigma_{\l,4,-3}  \ar@{=}[d] \ar@{--}[dl]&\Ind _{G_{\Q(i)}}^{G_\Q}[ (\rho_{f}\mid_{G_{\Q(i)}})\otimes \chi]\ar@{=}[d]\\
\sigma_{\l,4,-3} =\sigma_{\l,2,-3} \otimes \psi&\Ind_{G_{\Q(i)}}^{G_\Q} \sigma_{\l,4,-1}  \ar@{--}[r] & (\sigma_{\l,4,-1} )^{ss}
}}$$
 {For a number field $K$, let $G_K$ denote $\Gal(\overline \Q/K)$. In the diagram above,  both $\sigma_{\l,2,-3} $ and $\sigma_{\l,4,-3} $ are 2-dimensional representations of $G_{\Q(\sqrt{-3})}$, $\psi$ is a cubic character of $G_{\Q(\sqrt{-3})}$ and $\chi$ is a quartic character of $G_{\Q(i)}$ where $i=\sqrt{-1}$.

\subsection{}
\label{SS:intro6}

For each prime $p\ge 5$ we prove that there exists a basis
$h_p = \{ h_p ^+, h_p ^- \}$ of $S_3 (\Gamma )$ whose expansion coefficients
satisfy 3-term Atkin and Swinnerton-Dyer (ASD) relations. To be more precise, there exists a finite extension $E$ of $\Q_p$ such that the coefficients of $h_p^{\pm}=\sum_{n \ge 1} a_\pm(n)q^n\in \mathcal O_E[[q]]$
 satisfy $$ a_{\pm}(np^r)-A_{p,\pm}a_{\pm}(np^{r-1})+B_{p,\pm}a_{\pm}(np^{r-2})\equiv 0 \mod p^{r(k-1)}, \forall n,r\ge 1,$$
 where $\mathcal O_E$ is the ring of integers of $E$, $A_{p,\pm},B_{p,\pm}\in \mathcal O_E$
and the weight $k = 3$.
 We say that we have ASD congruences relative to the polynomial
$X^2 -A_{p, \pm}X +B_{p,\pm}$.
 The structure of $h_p$ depends only
on the class of $p$ mod 12. As an application of the modularity result mentioned above, $A_{p,\pm},B_{p,\pm}$ are determined by the $p$-coefficient of $f$ and the characters $\psi$ and $\chi$.
See Propositions \ref{prop:p=1 mod 3},  \ref{prop:p=5}, and \ref{prop:p=11}.

%Briefly speaking, the reason is as follows. The $\ell$-adic representations $\rho_{\ell,2}$ have $p$-adic crystalline analogs
%$V_p$, which are 4-dimensional $\Q _p$ vector spaces that contain
%$S_3 (\Gamma ) \otimes \Q _p$, and are acted upon by a Frobenius $F$.
%Depending on the class of $p$ mod 12, the correspondences $B_j$ induce
%automorphisms of $V_p$ commuting with $F$, so that $F$ acts
%on the 2-dimensional $B_j$-eigenspaces $V_p ^{\pm}$. Then $h_p^{\pm}$
%are defined as the generators of the 1-dimensional eigenspaces
%\[
 %[S_3 (\Gamma ) \otimes \overline {\Q}_p]\cap (V_p ^{\pm}\otimes \overline {\Q}_p).
%\]
%Since
%$\mathrm{Char} (\mathrm{Frob}_p, \sigma_{\lambda,2,j}^{\pm}, X) =
%\mathrm{Char} (F, V_p ^{\pm}, X)\in \overline{\Q}[X]$, and
%\[
 %\mathrm{Char} (F, V_p ^{\pm}, F) (h_p ^{\pm}) = 0
%\]
%by Cayley-Hamilton, the ASD relations follow by standard arguments
%in crystalline cohomology.

\subsection{}
\label{SS:intro7}The paper is organized as follows.
In \S \ref{S:group}, we describe the group $\Gamma$ and the
family of elliptic curves $E(\Gamma)\to X(\Gamma)$ associated to it.
In \S \ref{S:corr} we define some correspondences $B_j$ of the
elliptic surface $E(\Gamma)$. The main point is that these define
a quaternion multiplication structure on associated cohomology spaces.
In \S \ref{S:modular} we show that the 4-dimensional $\ell$-adic
representations $\rho_{\ell, 2}$, $\rho_{\ell, 4}$ are induced in several
ways from subgroups of index two in $G_{\Q}$. \S \ref{S:mod}
proves the main modularity theorem: the $L$-function of the representations
$\rho_{\ell, 2}^{\vee}$ (resp. $\rho_{\ell, 4}^{\vee}$) can be expressed in terms of the Hecke
 $L$-function of a newform $f$ of level 432 and some explicit characters. This is done by the method
 of Faltings-Serre. The ASD congruences are proved in
\S \ref{S:aswd}. Tables of experimental data which form the basis of this paper
 appear in \S \ref{S:tab}. These computations began in an REU project
 in summer 2005. The software systems \texttt{Magma}, \texttt{Mathematica}, and \texttt{pari/gp} were
 used.

\section{The group and the space of noncongruence cuspforms}
\label{S:group}
\subsection{}
\label{SS:group1}
If $\Gamma _0\subset \SL $ is a torsion-free subgroup of finite index we let
$Y(\Gamma _0) =\Gamma _0 \backslash \H $ be the quotient of the upper half plane of complex numbers,
and $j : Y(\Gamma _0) \to X(\Gamma _0)$ be the compactification by adding cusps. It is known that
these are the $\C$-points of algebraic curves defined over number fields; in this paper, they
will have models over $\Q$, which we will denote by the same symbols.
Define the analytic space $E(\Gamma _0)$ as the
quotient of $\C \times \H$ by the equivalence relation
\[
(z, \tau ) \sim \left (\frac{z+m \tau + n}{c\tau +d}, \frac{a\tau + b}{c\tau +d}\right ),
\quad m, n \in \Z, \ \ \gamma = \begin{pmatrix}
                                 a & b \\c& d
                                \end{pmatrix} \in \Gamma _0.
\]
Then $f: E(\Gamma _0) \to Y(\Gamma _0)$ is a fibration of elliptic
curves. When $\Gamma _0$ is a congruence subgroup, these are the
$\C$-points of schemes defined over number fields and represent (at
least coarsely) moduli problems for elliptic curves. In this paper,
they will be defined over $\Q$ and designated by the same symbols.

\subsection{}
\label{SS:group2}
(For generalities on the moduli spaces of elliptic curves, see \cite{DR}, \cite{km}).
The stack $[\Gamma _0 (8)]$ classifies pairs $(E, C)$ of elliptic curves
$E$ together with subgroup schemes $C\subset E$ locally isomorphic to $\Z /8$. Since
$\pm 1 \in \Gamma _0 (8)$ the map $[\Gamma _0 (8)] \to M(\Gamma _0 (8))$ is two to one, { where
for a congruence subgroup $\Gamma _0$, $M(\Gamma _0)$ denotes the corresponding
(coarse) moduli scheme}. One knows that
$M(\Gamma _0 (8))\otimes \Q  \cong \mathbf{P} ^1 _{\Q} = X(\G_0(8))$.

\subsection{}
\label{SS:group3}
The stack $[\Gamma _1 (4)]$ classifies pairs $(E, P)$ of elliptic curves
$E$ together with a point $P$ of exact order 4. This time
$[\Gamma _1 (4)] = M(\Gamma _1 (4))$. One knows that there are two connected components
defined over $\Q(i)$ each of which is isomorphic to $\mathbf{P} ^1 _{\Q (i)}$.

\subsection{}
\label{SS:group4}
The stack $[\Gamma _0 (8)\cap\Gamma _1 (4)]$ classifies triplets
$(E,C, P)$ of elliptic curves $E$ together with $P\in C$ a point of
exact order 4 in a  cyclic subgroup of order 8. We have $[\Gamma _0
(8)\cap \Gamma _1 (4)] =  M(\Gamma _0 (8)\cap \Gamma _1 (4))$. In
fact, projectively  $\pm \Gamma _0 (8)/\pm I_2 =\pm(\Gamma _0 (8)\cap \Gamma _1 (4))/\pm I_2$
so the modular curves are the same:
$M(\Gamma _0 (8)\cap \Gamma _1(4))\otimes \Q
 =  M(\Gamma _0 (8))\otimes \Q\cong \mathbf{P} ^1 _{\Q}$.
It  has a
fine moduli interpretation as the moduli space of triples $(E, C,
P)$ where $E$ is an elliptic curve, $C\subset E$ is a subgroup scheme
of order 8, and $P\in E$ is a point of order 4. A model for its
universal elliptic curve is
\begin{equation}\label{eq:1}
E_8 (t) : \quad y^2 + 4xy + 4t^2y =x^3 + t^2x^2,
\end{equation} where $t=\frac{\eta(z)^8\eta(4z)^4}{\eta(2z)^{12}}
=1-8q+32q^2+\cdots\in \Z[[q]], q=e^{2\pi i z}$ cf. \cite{l5}. The modular function $t$ is  a Hauptmodul of
$\Gamma _0 (8)\cap \Gamma _1 (4)$: a generator
of the function field of the modular curve.  Let
$\G$ be a special index 3 normal subgroup of $\Gamma _0 :=\G_0(8)\cap \G_1(4)$
whose modular curve $X(\G)$ is a cubic cover of  $X(\G_0(8)\cap \G_1(4))$ unramified everywhere except
the cusps $\frac 12$ and
$\frac 14$ (whose $t$-values are $\infty$ and $-1$ respectively),
with ramification degree 3. It is easy to see that the genus of
$X(\G)$ is also 0. To facilitate our calculation, we need to find an
algebraic map between the two modular curves, in  other words, we
need to describe a relation between a Hauptmodul $r_a$ of $\G$ and
$t$. By our construction, $a r_a^3=t+1$ for some nonzero constant
$a$. Here $a$ is a rational number written in lowest form. As we
look for a model so that the $q$-expansion of $r_a$ is in a number field
as small as possible, we take $a = 2$ in which case the coefficients of $r_a$ can be chosen in $\Q$.
The Riemann surface $ Y(\Gamma):=
\Gamma \backslash\H$ has the structure of an algebraic curve over $\Q$. This group
$\G$ is labeled by $\G_{8^3\cdot6\cdot 3\cdot 1^3}$ in \cite{l5}. The cusp widths  of $\Gamma$ are 8-8-8-6-3-1-1-1 from which we know $\G$
is a noncongruence subgroup.
The quotient $\G_0/\G$ is generated by $\zeta \G$, where $ \zeta=\begin{pmatrix}
5&-2\\8&-3
\end{pmatrix}$.  The matrix $A=\begin{pmatrix}
  0&-1\\8&0
\end{pmatrix}$ normalizes $\G$. Both matrices $\zeta$ and $A$ induce actions on $X(\G)$ as well as the spaces of cuspforms for $\Gamma$.

\subsection{}
\label{SS:group5}
Note that in \cite{l5}, different choices of $a$ are picked.
% Any two different choices give isomorphic Galois representations of an explicit finite index subgroup of $G_\Q$.
The reason for varying $a$ is that
for some choices, the operators to be defined below had smaller
fields of definition, whereas for other choices, the Galois representation corresponding to $S_3(\G)$ and parameter $a$ (see  \S \ref{S:modular})
was easier to analyze. Only two choices $a = 2, 4$ are relevant to this
paper; the corresponding Hauptmoduls are denoted $r_2$ and $r_4$.
 We
let $f_a: E(\Gamma) \to Y(\Gamma)$ be the pull-back of the universal
elliptic curve in the previous section via the covering of degree
three:
\begin{eqnarray*}
 Y(\Gamma)&\to&  Y(\G_0(8)\cap \G_1(4))\\
 r_a&\mapsto& a r_a^3 -1
\end{eqnarray*}
We let $\E _{\G}$ be the complex elliptic surface obtained by
completing and desingularizing $f: E(\G)\to Y(\Gamma) $ over the
compact curve $X(\Gamma):= (\Gamma \backslash \H)^*$, which is independent of
the choice of $a$. The nonzero Hodge numbers of this surface are
$h^{0, 0} = h^{2, 2} = 1$, $h^{1, 1} = 30$, $h^{2, 0} = h^{0, 2} =
2$. In particular the space of weight three cuspforms
\[
 S_3(\G) = H^0 (\Omega ^2 _{ \E _{\G}/\C})
\]
is two dimensional. This $S_3(\G)$ has a
basis (see  \cite{l5}): \begin{equation}\label{eq:h1&2} h_1=\sqrt[3]{H_1}, \
H_1:={\frac{\eta(z)^4\eta(2z)^{10}\eta(8z)^8}{\eta(4z)^4}}, \quad
h_2=\sqrt[3]{H_2}, \
H_2:={\frac{\eta(z)^8\eta(4z)^{10}\eta(8z)^4}{\eta(2z)^4}}.
\end{equation}

\section{Correspondences}
\label{S:corr}
\subsection{}
\label{SS:corr1}
Given two subgroups of finite index $\Gamma_1, \Gamma_2\subset \mr{SL}_2 (\Z)$ and an element
$\alpha \in  \mathrm{M}_2( \Z)$ with $\det (\alpha ) > 0 $  the double coset
$\Gamma_1 \alpha \Gamma_2$ determines
a correspondence $\Gamma _2 \backslash \mathfrak{H} \longrightarrow \Gamma _1 \backslash \mathfrak{H}$.  Namely, we have a diagram
\[
    \Gamma _2 \backslash \mathfrak{H}  \leftarrow \Gamma \backslash \mathfrak{H}
\rightarrow \Gamma _1 \backslash  \mathfrak{H}
\]
where $\Gamma = \Gamma_2 \cap \alpha ^{-1} \Gamma_1 \alpha $. The first arrow sends
$\tau \in \mathfrak{H}$ to $\tau$ and the
second arrow sends $\tau$ to $\alpha.\tau$. If   $\Gamma_2$ is decomposed into cosets for the
subgroup $\Gamma $
\[
    \Gamma_2 = \coprod _{i=1}^{d} \Gamma . \varepsilon_i,
\]
so that the $\varepsilon_i. \tau$ are the elements of $\Gamma \backslash \mathfrak{H}$ projecting to
$\tau \in \Gamma _2\backslash\mathfrak{H}$, then the double coset decomposes as
\[
    \Gamma_1 \alpha \Gamma_2  = \coprod _{i=1} ^{d} \Gamma_1 . \alpha \varepsilon_i :=
\coprod _{i=1} ^{d} \Gamma_1 . \alpha_i
\]
and the correspondence sends the point $\tau$ mod $\Gamma_2$ to the cycle
\[
   \sum _{i = 1}^{d} \alpha_i \tau \mathrm{\ mod\ } \Gamma_1.
\]
\subsection{}
\label{SS:corr2}
 This diagram lifts to morphisms
\[
    E(\Gamma_2) \stackrel{p_2}{ \leftarrow} E(\Gamma )\stackrel{p_1}{\rightarrow} E(\Gamma_1)
\]
The arrow on the left is induced by $(z, \tau) \to (z, \tau)$. This is a
fiberwise isomorphism. The map on the right is
induced from
\[
  (z, \tau)  \mapsto   (\det(\alpha ).z/j(\alpha, \tau), \alpha.\tau).
\]
This is a fiberwise isogeny.

\subsection{}
\label{SS:corr3}
The double coset $\Gamma _1 \gamma \Gamma _2$ induces maps on cohomology
via
$$p_{2 *}\circ p_1 ^*: H^i (\E _{\Gamma _1}^{k} , \Q) \to H^i(\E _{\Gamma _2}^k, \Q )$$
for any integer $k \ge 1$, where  $\E _{\G}^{k}$ denotes the desingularization of the
$k$-fold fiber product of $\E_{\G}/X(\G)$. Recall that there is a canonical injection
$S_{k+2} (\Gamma) \hookrightarrow H^1 (\E _{\Gamma }^{k} , \C)$ for $k \ge 1$
which identifies it with the $(k+1, 0)$ part of the Hodge structure of pure weight
on the right-hand side. The induced map given by the double coset
$S_{k+2} (\Gamma _1) \to S_{k+2}(\Gamma _2)$ is given by the slash operator:
\[
 f\mid [\Gamma _1 \alpha \Gamma _2]_{k+2} = \det (\alpha )^{k/2 }\sum _{i = 1}^d  f \mid [\alpha _i]_{k+2}
\]
see  \cite[Ch. 3, especially 3.4]{Shi71}.

\subsection{}
\label{SS:corr4}
{ We apply the above to the group $\Gamma_1 = \Gamma _2 = \Gamma _0 (8)$
and the matrix
$$\alpha  = A = \begin{pmatrix}
                       0 & -1 \\ 8 & 0
                          \end{pmatrix}
$$
which normalizes $\Gamma _0 (8)$. Then $k = 1$, $d = 1$ and $\alpha _1 = \alpha$, $p_2 = \text{id}$ so
that the map on cohomology is given by $p_1 ^*$.}
The involution of the modular curve
$Y(\Gamma _0(8))$ induced by this matrix has the moduli interpretation
$(E, C) \mapsto (E/C, E[8]/C)$, where $(E, C) \in Y(\Gamma _0(8)) $ is an
elliptic curve with a cyclic subgroup of order $8$, and $E[8]$ is the
kernel of multiplication by $8$. Recall that $t$ was our chosen Hauptmodul
for the curve $Y(\Gamma _0(8))$.
By a direct calculation, we know $A=\begin{pmatrix} 0&-1\\8&0
\end{pmatrix}$ induces the following map on $X(\G_0(8)\cap \G_1(4))$ (see \cite{l5}):
\begin{equation}\label{eq:AonX}
  t\mapsto\frac{1-t}{1+t}.
\end{equation}
{In this case, the map $p_1^*$ factors
\[
 E_8 (t) \overset{A'}{\longrightarrow}A^* E_8(t) = E_8 \left (\frac{1-t}{1+t}\right )
 \overset{A''}{\longrightarrow} E_8 (t).
\]
Here $A''$ is an isomorphism defined over $\Q$, because it is the base-change
along $E(\Gamma _0 (8)) \to Y(\Gamma _0 (8))$
of the involution $A$ of $Y(\Gamma _0(8))$,  which is an isomorphism defined over $\Q$. Then
$p_1^*$ is an isogeny defined over $\Q (i)$ covering
the automorphism of $\Q (t)$ given by $A$  because of:
}
\begin{prop}
\label{P:corr1}
%The operator $A$ acts on the modular curve $X(\Gamma _0(8) \cap
%\Gamma_1(4))$ via $t \mapsto \frac{1-t}{1+t}$.
This action of $A$  on  $X(\Gamma _0(8) \cap
\Gamma_1(4))$  lifts to
an isogeny $A': E_8(t) \to E_8 \left (\frac{1-t}{1+t}\right )$ defined over $\Q (i,
t)$.
\end{prop}
\begin{proof}
It
can be shown that there is no isogeny $E_8(t) \to E_8
(\frac{1-t}{1+t})$ defined over $\Q (t)$. One way to see this is by
specializing $t$ to have rational values. These calculations were carried out using Magma and Mathematica.
First one calculates the quotient curve $E_8 (t)/C$ where $C$ the
universal subscheme giving the cyclic group of order $8$. This is
the subgroup scheme defined by $(x^2 -4tx-4t^3)(x+t^2)x = 0$. This
gives the curve $E_8 '(t):  y^2 + 4xy + 4t^2y  = x^3 + t^2x^2 +
b(t)x + c(t)$, where $b(t) = -5t^4 - 320t^3 - 720t^2 - 320t,$ and $c(t) = 3t^6 - 704t^5
- 5184t^4 - 8896t^3 -5888t^2 - 1024t$,
%\begin{eqnarray*}
%&&b(t) = -5t^4 - 320t^3 - 720t^2 - 320t, \\
%&&  c(t) = 3t^6 - 704t^5
%- 5184t^4 - 8896t^3 -5888t^2 - 1024t
%\end{eqnarray*}
and an explicit isogeny $\psi : E_8 (t)\to E_8 '(t)$ defined over
$\Q (t)$. Next one constructs an isomorphism $\phi : E_8 '(t) \to
E_8 (\frac{1-t}{1+t})$. It is
\[
  (x, y) \mapsto (\phi _1(x, y), \phi _2 (x, y))
\]
where
\begin{eqnarray*}
&&\phi_1 (x, y) = -\frac{7 t^2+8 t+x+8}{4 (t+1)^2},\\
&&\phi_2 (x, y) = \frac{12 t^3+2 (i+38) t^2+4 (x+20) t+2 (i+2) x+i
y+16}{8
   (t+1)^3}.
\end{eqnarray*}

The isogeny is $\phi\circ \psi$. More explicitly, we can write $\psi
= \psi'''\circ \psi''\circ \psi'$ where
\begin{align*}
& \psi '_1 (x, y) = \frac{t^4+x t^2+x^2}{t^2+x}, \quad  \psi'_2 (x,
y) = \frac{-4 t^6-4 x t^4+2 x y t^2+x^2 y}{\left(t^2+x\right)^2}
 \\
& \psi ''_1 (x, y) =\frac{t^2 (x-16)-x^2}{t^2-x}, \quad  \psi''_2
(x, y) = \frac{y t^4-2 (8 y+x (y+32)) t^2+x^2
y}{\left(t^2-x\right)^2},
& \\
& \psi '''_1 (x, y) = \frac{-64 t^3+(x-128) t^2+8 (x-8) t-x^2}{t^2+8
 t-x}\\
& \psi'''_2 (x, y) = \tiny {\frac{P(x,y)}{\left(t^2+8 t-x\right)^2}},
\end{align*}$P(x,y)=(y+1024) t^4-16 (16 x+3 y-128)
t^3-2 (32 (y-16)+x (y+256))
   t^2-16 (4 y+x (y+16)) t+x^2 y$.
\end{proof}

\subsection{}
\label{SS:corr5}
The involution $A$ of the modular curve $X(\Gamma _0(8)\cap \G_1(4))$  lifts to an involution
of the curve $X(\Gamma) $ where $\Gamma\subset \Gamma _0 (8)\cap \G_1(4)$ is defined in \S \ref{SS:group4}.  %for the noncongruence subgroup $\Gamma \subset \Gamma _0 (8)\cap \G_1(4)$ defined in \ref{SS:group4}.
Let $r_4$ be the Hauptmodul for this curve with $t = 4 r_4^3 -1$. Under the action of $A$,
$r_4 \mapsto 1/2r_4$.  By base-change, the isogeny $A$ defined in the above proposition
lifts to an isogeny of the elliptic curve $E(\Gamma )$, which we will denote by $E(r_4)$.
This map defined over $\Q (i)$.
For the purposes of the ASD congruences, we need that the curve $X(\Gamma )$ has a
$\Q$-rational cusp corresponding to $\tau = i \infty$; unfortunately, this is not the
case for this model: the point $\tau = i\infty $ corresponds to $t = 1$ and there is no
$\Q$ -rational $r_4$ with $t = 1 = 4r_4^3 -1$. In the model
with $t = 2r_2 ^3 -1$ corresponding to the representation
$\rho _{\ell, 2}$ (details see \S \ref{S:modular}), there is a $\Q$-rational point over $t=1$, namely $r_2=1$. But in this model,
the involution $A$ is now given by
$r_2 \mapsto 1/\sqrt[3]{2}r_2$, and this gives an isogeny of $E(\Gamma) $ defined
over the larger field $\Q (i, \sqrt[3]{2})$.  In either of these models,
$\zeta $ acts on the Hauptmodul by  $r_a \mapsto \exp{(2 \pi i/3)}\, r_a$.

\subsection{}
\label{SS:corr6}
On the $p$-adic space $V_p$ (see sections \ref{SS:intro2} and \ref{SS:aswd1}), we define operators
\[
B_{-1}  = A, \quad B_{-3} = \zeta - \zeta ^2, \quad B_3  = A(\zeta - \zeta ^2),
\]
where $A$ and $\zeta$ are given in \S \ref{SS:group4}.

\begin{prop}
\label{P:padic}
  $B_{-1}^2 = -8$;   $B_{-3}^2 = -3$;  $B_{3}^2 = -24$;
 $B_{-1}B_{-3} = - B_{-3}B_{-1} $.
\end{prop}
\begin{proof}
To prove the identities, it suffices to prove them for the
corresponding operators on $S_3 (\Gamma)$.
{
 The reason is that this is the Hodge $(2, 0)$
component of $H _{DR} (\phantom{}_{\Gamma }^3\mathcal{W})$ for Scholl's motive
(see \cite{Sch85iii}), these operators act on this motive, which is
a natural factor of $H^2 (\E _{\G})$, and $V_p = H_{cris} (\phantom{}_{\Gamma }^3\mathcal{W})$.}
\footnote{Scholl only constructs $\phantom{}_{n }^k\mathcal{W}$ for principal congruence subgroups
of level $n$. But the construction also works here: it is the Grothendieck motive which is
formal image of the projector denoted $\Pi _{\varepsilon}$ in loc. cit. acting on the elliptic surface
$\E_{\G, a}$, now regarded as a scheme over $\Q$.}
The effect of $\zeta $ on the basis $h_1, h_2$ (see \eqref{eq:h1&2}) is given by the matrix
\[
\zeta:=\begin{pmatrix}
  \omega_3 &0\\0&\omega_3^{-1}
\end{pmatrix}
\]
where $\omega_3$ is a primitive cubic root of unity. To compute the effect of
$A$ we  use the stroke operator defined in section (3.3) on $H_1 = h_1 ^3$, $H_2 = h_2 ^3$,
and use
properties of the $\eta$-function. The result is, in the basis $h_1, h_2$
\[
 A:=\begin{pmatrix}   0& i 2^{4/3}\\
i 2^{5/3}&0
\end{pmatrix}
\]
These identities follow immediately.
\end{proof}

\section{$\ell$-adic representations}
\label{S:modular}
In this paper, for any place $v$ of a number field $K$, we use $\text{Fr}_v$ and $\Frob_v:=\text{Fr}_v^{-1}$  to denote the corresponding arithmetic Frobenius and geometric Frobenius, respectively. We use $\ell$ to denote an arbitrary prime, unless it is specified.

\subsection{}
\label{SS:modular0}
If $j: Y(\G) \to X(\G) $ is the inclusion, we define
\[
 W_{\ell,a} = H^1 (X(\Gamma) \otimes \overline{\Q}, \  j_* R^1 f_{a*} \Q _\ell)
\] which is a 4-dimensional $\Q_{\ell}$-space,
where \'etale cohomology is understood, and we are using the symbols
$ Y(\G)$ and $ X(\G)$ to denote the schemes over $\Q$ whose $\C$-valued points were
previously denoted by these letters. We thus obtain a continuous $\ell$-adic representation
\[
\rho_{\ell,a} : G_{\Q} \to \mathrm{GL}(W_{\ell,a}) = \mathrm{GL}_4( \Q
_\ell).
\]
Let $N(a)$ be the least common multiple of the
numerator and denominator of $a$. This representation is unramified
outside of $2,3,\ell, N(a)$, and for all primes $p\nmid 6\ell N(a)$, we
have the characteristic polynomial of Frobenius
\[
\det (X - \rho _{\ell,a} (\Frob _p)  )=\text{Char}(\Frob_p, W_{\ell,a},X) = H_{p,a}(X) \in \Z [X].
\] A useful fact is that the roots of  $H_{p,a}(X)$ have the same absolute value (cf. \cite{Sch85ii}),
which is $p$ in this case.

\subsection{}
\label{SS:modular.5}
When $a=4$, it was first observed experimentally  that there are factorizations
\[
H_{p, 4}(X) = g_{p,4}(X) \overline {g_{p,4}(X)},
\]
where the bar notation stands for complex conjugation and the coefficients of the quadratic polynomials $g_{p,4}(X)$ lie
respectively in the fields $\Q$, $\Q (\sqrt {-3})$,
$\Q (\sqrt {-2})$, $\Q (\sqrt {-6})$ as $p \equiv 1, 7, 5, 11$
mod $12$. This is a property that follows from $\rho_{\ell,4}$ satisfying the so-called quaternion multiplication over the biquadratic field $\Q(\sqrt{-3},\sqrt{-2})$ in the sense of \cite[\S 3]{ALLL09}. To be more precise, we consider the following maps on $W_{\ell,4}\otimes_{\Q_\ell}\overline{\Q}_\ell$. Let $\zeta^*$ be the map on $W_{\ell,4}$ induced by $\zeta$ and $A^*$ be the map induced by $A$ (see \S 3.4 and a related discussion
in \cite[\S 5]{ALLL09}).
It sends $E(\G)$ to an isogenous elliptic curve over $\Q(i,r_4)$. It is obvious that $\zeta^*$ is defined over $\Q(\sqrt{-3})$. By \S \ref{SS:corr4}, $A^*$ is defined over $\Q(i)$. We define the following operators on $W_{\ell,4}\otimes_{\Q_\ell}\overline{\Q}_\ell$: $B_{-1}^*=A^*, B_{-3}^*=\zeta^*-(\zeta^*)^2$; as well as \[
J_{-1}  = \frac 1{\sqrt 8} B_{-1}^*, \quad  J_{-3} = \frac 1
{\sqrt{3}}B_{-3}^*, \quad J_3 = J_{-1} J_{-3}
\]
These depend on choices of embeddings $\sqrt{8}, \sqrt{3} \in \overline{\Q}_\ell$,
but the decompositions that follow below do not depend on these choices.
\begin{prop}
 \label{P:corr2}
%On $W_{\ell,4}\otimes_{\Q_\ell}\overline{\Q}_\ell$,
%We have:
1. $J_{-1}^2 = -1$;   $J_{-3}^2 = -1$;
 $J_{-1}J_{-3} = -J_{-3}J_{-1}$.

2. {When $a = 4$, for $s \in \{-1, -3, 3 \}$},  $J_s \rho_{\ell, 4} = \varepsilon _s \rho_{\ell, 4} J_s$ where
$\varepsilon _s : G_{\Q} \to \C ^*$ is the quadratic character of $G_\Q$ whose kernel is $ G_{\Q (\sqrt{s})}$.

%\end{enumerate}

\end{prop}
\begin{proof}
%{\color{blue} Should be rewritten} {\color{green} (What would you recommend? -Ling)}
%Let $B_{-1}^*, B_{\pm 3}^*$ be the maps on $W_{\ell,4}\otimes_{\Q_\ell}\overline{\Q}_\ell$ induced by $B_{-1}, B_{\pm 3}$ as before.
(1) In the proof of Proposition \ref{P:padic} these identities were shown
to hold on   $H _{DR} (\phantom{}_{\Gamma }^3\mathcal{W}) $. By comparison isomorphisms, these also
hold on $H _{\ell}(\phantom{}_{\Gamma }^3\mathcal{W} )= W_{\ell, a}$.

(2) The Galois group $G_{\Q}$ interacts with the operators $J_{\pm 3}, J_{-1}$ in the same
way that it interacts with the $B_{-1}^*, B_{-3}^*, B_3^*$, in other words the irrationalities
  $\sqrt{8}, \sqrt{3}$ do not affect the Galois action: the reason is that
one treats the second factor in  $W_{\ell, 4}\otimes _{\Q_\ell}\overline{\Q}_\ell$
as the trivial Galois module.  Thus, {when $a = 4$,} by the discussion right before the proposition, $\Q(i)$ (resp. $\Q(\sqrt{-3})$,
{  $\Q(\sqrt{3})$}) is the minimal field of definition of $J_{-1}$ (resp. $J_{-3}$, { $J_3$}).  Thus the commutativity of $G_{\Q}$ and $J_{-1},J_{-3}, J_3$ is as claimed.
\end{proof}

\begin{cor}
 \label{C:modular1}
1. Let $D_{-1} = -2$,  $D_{-3} = -3$,  $D_{3} = -6$.
For each  $s\in \{-1, -3, 3 \}$, let  $K_s = \Q(\sqrt{D_s})$ and  $\lambda$ be the place of
$\Q _\ell (\sqrt{D_s})= K_{s, \lambda}$ lying over $\ell$. Then
{
\begin{equation}\label{eq:6}
 \rho_{\ell, 4}\otimes _{\Q_\ell}K_{s, \lambda} =
\mathrm{Ind}_{G_{\Q (\sqrt{s})}}^{G_{\Q}} (\sigma_{\lambda, 4, s} ), 
\end{equation}
}
for some 2-dimensional absolutely irreducible representation $\sigma_{\lambda, 4, s} $ of
{$G_{\Q (\sqrt{s})}$}.

2.  The determinant $\det \sigma_{\lambda, 4, s} =\varphi_s \cdot (\epsilon_\ell|_{G_{\Q(\sqrt s)}})^{2}$ where $\epsilon_\ell$ is the $\ell$-adic cyclotomic character and $\varphi_s$ is the quadratic character of $G_{\Q(\sqrt s)}$ with kernel $G_{\Q(i,\sqrt 3)}$.
\end{cor}
\begin{proof}
1. 
For each $s$ the eigenspaces of $B_s ^*$ are
each two dimensional, defined over $K_s$, invariant under $G_{\Q (\sqrt{s})}$. Thus $\rho_{\ell, 4}\otimes _{\Q_\ell}K_{s, \lambda}|_{G_{K_s}}=\sigma_{\lambda, 4, s} \oplus \sigma_{\lambda, 4, s} '$ where $\sigma_{\lambda, 4, s}$ and its conjugate $\sigma_{\lambda, 4, s} '$ are 2-dimensional. It is straightforward to check, by using the data listed in Table \ref{T:moddata2}, they are absolutely irreducible and non-isomorphic. Thus \eqref{eq:6} follows from Clifford's result in \cite{Clifford37}. Moreover, by \cite{Clifford37} one knows $\sigma_{\lambda, 4, s}= \sigma_{\lambda, 4, s} '\otimes \chi_s$ where $\chi_s$ is the quadratic character of $G_\Q$ with fixed field $K_s$.

2. By a direct calculation, one knows that $\det \rho_{\ell, 4}=\epsilon_\ell^{4}$. It follows from $\sigma_{\lambda, 4, s}= \sigma_{\lambda, 4, s} '\otimes \chi_s$ that  $\det \sigma_{\lambda, 4, s} =\varphi_s \cdot (\epsilon_\ell|_{G_{\Q(\sqrt s)}})^{2}$ for some character $\varphi_s$ of {$G_{\Q(\sqrt{s})}$} of order at most 2. From the data (Table \ref{T:moddata2}) and the fact that $\varphi_s$ only ramifies at places of {$\Q(\sqrt{s})$} above 2 and 3, we can conclude that the fixed field of $\varphi_s$ is $G_{\Q(i,\sqrt 3)}$.\end{proof}

\begin{theorem}\label{thm:mod-main-1}
 \label{C:modular2}
1. The semi-simplification of $\rho _{\ell, 4}^{\vee}$, the dual   of $\rho _{\ell, 4}$, is automorphic, i.e.
the $L$-function of $\rho _{\ell, 4}^{\vee}$ is equal to the $L$-function of an automorphic
representation of $\text{GL}_4(\mathbb A_\Q)$.

2. To be more precise, $L(s, \rho _{\ell, 4}^{\vee})=L(s, (\rho_f|_{G_{\Q(i)}})\otimes \chi)$ for some level 432 Hecke newform $f$ and some quartic character $\chi$ of $G_{\Q(i)}$.
\end{theorem}
The first claim follows from 4.2.4 and 4.2.5 of \cite{ALLL09}. We will postpone the proof of the second claim to the next section. See Theorem \ref{thm:mod-main}.}
\subsection{}
When $a=2$, the operator $B_{-3}^*$ on $\rho_{\ell,2}$ is defined over $\Q(\sqrt{-3})$.   Because of it, $\rho_{\ell,2}\otimes_{\Q_\ell} K_{-3,\l}$ is also induced from a 2-dimensional representation $\sigma_{\l,2,-3} $ of $G_{\Q(\sqrt{-3})}$ (see \cite{long061}). The factorization of  $H_{p,2}(X)$ is given in Table \ref{T:moddata}.
 {
\begin{theorem}\label{thm:mod-main-2}
Let $K$ be the splitting field of $x^3-2$ over $\Q$. Then
  $\sigma_{\l,4,-3} =\sigma_{\l,2,-3} \otimes \psi$ where  $\psi$ is a cubic character of $G_{\Q(\sqrt
{-3})}$ with kernel $G_K$.
\end{theorem}
\begin{proof}
Note that  $\rho_{\ell,2}\mid_{G_K}$ and $\rho_{\ell,4}\mid_{G_K}$ are isomorphic as { the
corresponding elliptic modular surfaces become isomorphic over $K$}. It is routine to check that $\sigma_{\l,2,-3} \mid_{G_K}$ and $\sigma_{\l,4,-3} \mid_{G_K}$ are absolutely irreducible. Upon replacing $\sigma_{\l,2,-3}$ by its conjugate under any element in $G_\Q\setminus G_{\Q(\sqrt{-3})}$, we may assume that $$\sigma_{\l,2,-3}|_{G_K}=\sigma_{\l,4,-3}|_{G_K}.$$
By Theorem 5 of \cite{Clifford37},  $\sigma_{\l,2,-3} $ and  $\sigma_{\l,4,-3} $  are either isomorphic or differ by a cubic character $\psi$ of $G_{\Q(\sqrt{-3})}$ with kernel $G_K$. Numerical data in Tables \ref{T:moddata} and \ref{T:moddata2} reveal that they  are not isomorphic.
\end{proof}}
\begin{lemma}\label{cor:2}
  For any prime $p\equiv2 \mod 3$ (which is inert in $\Z[\frac{1+\sqrt{-3}}2]$), $\psi(p)=1$.
\end{lemma}
\begin{proof}
 The polynomial $X^3-2$ has exactly one root in $\F_p$ when $p\equiv2 \mod 3$ as $X\mapsto X^3$ is a bijection in $\F_p$. Thus for any prime  $p\equiv2 \mod 3$, which is inert in $\Z[\frac{1+\sqrt{-3}}2]$, it splits completely in the Galois extension $K$ of $\Q(\sqrt{-3})$. This means $\psi(p)=1$.
\end{proof}

\section{Modularity}\label{S:mod}
Our goal is to prove Theorem \ref{thm:mod-main-1}.  Since $\rho_{\ell,4}$ satisfies quaternion multiplication over  $\Q(i,\sqrt{-3})$,   by the main result of \cite{ALLL09} the L-function of $\rho_{\ell,4}^{\vee}$ coincides with the L-function of a $\text{GL}_4(\mathbb A_{\Q})$ automorphic form. Here we will prove  this claim directly.
\footnote{ Using the approach of \cite{ALLL09}, one can  also derive that there  exists  a quadratic character $\phi$ of $G_{\Q(\sqrt{-3})}$ with fixed field $\Q(\sqrt[4]{-3})$ such that $ (\sigma_{\l,2,-3} \otimes \psi\otimes\phi)^{\vee}=(\sigma_{\l,4,-3} \otimes \phi)^{\vee}$ agrees with $\rho_g|_{G_{\Q(\sqrt {-3})}}$ where $\rho_g$ is the Deligne representation attached to a certain new form $g$, that is different from the new form $f$ below.}

We identified using William
Stein's Magma package the following Hecke eigenform
\begin{eqnarray*}
f(z)&=& q + 6\sqrt{2}q^5 + \sqrt{-3}q^7 + 6\sqrt{-6}q^{11}+
 13q^{13} - 6\sqrt{2}q^{17} + 11\sqrt{-3}q^{19} \\&& - 18\sqrt{-6}q^{23}+
  47q^{25} - 24\sqrt{2}q^{29} +   24\sqrt{-3}q^{31} +
6\sqrt{-6}q^{35} + \cdots,
\end{eqnarray*}
More concretely,
\begin{equation}
f(z)=\sum_{n\ge 1}c_p(f)q^n = f_1(12z) + 6\sqrt{2} f_5(12z) + \sqrt{-3}f_7(12z) +
6\sqrt{-6}f_{11}(12z), \label{eqn:f_432_as_eta_and_einstein}
\end{equation}
where
\begin{eqnarray*}
f_1(z) =\frac{\eta(2z)^3\eta(3z)}{\eta(6z)\eta(z)}E_6(z), && f_5(z)=
\frac{\eta(z)\eta(2z)^3\eta(3z)^3}{\eta(6z)}
\\
f_7(z)=\frac{\eta(6z)^3\eta(z)}{\eta(2z)\eta(3z)}E_6(z),&&
f_{11}(z)= \frac{\eta(3z)\eta(z)^3\eta(6z)^3}{\eta(2z)}
\end{eqnarray*}
where $E_6(z)= 1 + 12\sum_{n\ge 1}(\sigma(3n) -3\sigma(n))q^n,$ and
$\sigma(n)=\sum_{d|n} d$. Let $\rho_{f}$ be Deligne's 2-dimensional  $\l$-adic representation of $G_{\Q}$ attached to $f$,
where $\lambda$ is the place of
$\Q _{\ell}(\sqrt{2}, \sqrt{-3})$ lying over $\ell$ (see \cite{D}, \cite{L}).  In particular
 the trace of arithmetic Frobenius $\text{Fr}_p$ under $\rho_f$ agrees with $c_p(f)$, the $p$-coefficient of $f$. These are the
conventions of \cite{DS}. { This is at variance with the conventions in \cite{D} and in Scholl's papers,
which match the characteristic polynomials of the geometric Frobenius with the Hecke polynomials
Therefore we must take duals in the statement of our main results.}

 We will now apply the method of Faltings-Serre (see \cite{serre}) to prove the result below. Briefly,
the idea is that given two nonisomorphic semisimple Galois representations
$\rho _1, \rho _2 : G_K \to \mr{GL}_2 (\Q _{\ell})$ there is a finite list of ($\tilde{G}, t$) that captures the difference between $\rho _1, \rho _2$,  where
each $\tilde{G}$ is a finite Galois group referred to as a {\it deviation group}, and
$t: \tilde{G} \to \F _{\ell}$ is a function with certain properties
that can be computed from $\rho _1, \rho _2$. To establish the isomorphism between the semisimplifications of $\rho _1$ and $ \rho _2$ we need to eliminate each  ($\tilde{G}, t$) by explicit computation. 
 The idea of the criterion is recast in \cite{lly05}, right below
the statement of Theorem 6.2.

%\footnote{ {\color{red}Then if
%$\rho _1 \neq \rho _2$ then there is a Galois extension of $K$
%with group $\tilde{G}$ and same ramification set as
%$\rho _1$ and $\rho _2$ such that
%$t (\mr{Frob} _p) = \rho _1 (\mr{Frob} _p) - \rho _2 (\mr{Frob} _p)$
% mod $\ell$.} Well, $t$ is a bit more complicated than that as it is explained in [LLY], page 136. Maybe we can skip %it? -Ling} %Since there are only finitely many such Galois extensions
%to check, 

\begin{theorem}\label{thm:mod-main} Let $\sigma_{\lambda, 4, -1}$ as before be the
representation of $G_{\Q(i)}$ whose induction to $G_\Q$ is $\rho_{\ell,4}\otimes _{\Q_\ell} K_{-1,\l}$. There exists   a quartic  character $\chi$ of $G_{\Q(i)}$ which fixes $L=\Q(i,\sqrt[4]3)$ such that up to semisimplification,  $\rho_{f} \mid _{G_{\Q(i)}}\otimes \chi$ is isomorphic to $(\sigma_{\lambda, 4, -1})^{\vee}$, the dual of $\sigma_{\lambda, 4, -1}$, as $G_{\Q(i)}$ representations.
\end{theorem}
\begin{proof}
The determinant  $\det  (\rho_{f} \mid _{G_{\Q(i)}}\otimes \chi)=\chi^2\varepsilon^2=\varphi_{-1}\varepsilon^2$, $\varphi_{-1}$ is the quadratic character of $G_{\Q(i)}$ with kernel $G_{\Q(i,\sqrt 3)}$. It coincides with the determinant of $(\sigma_{\lambda, 4, -1})^{\vee}$ by Corollary \ref{C:modular1}.

Let $H={G_{\Q(i)}}$.
  By the explicit form of $f$, $\rho_{f} \mid _{H}$ is a 2-dimensional representation of $H$ over $\Q_{\ell}(\sqrt 2)$. For any places $v$ above $p\equiv 5 \mod 12$, the character $\chi$ takes values $\pm i$ and  the characteristic polynomial of $\sigma_{\l,4,-1}(\Frob_v)$ is of the form $X^2+ a \sqrt{-2} X-p^2$ for some $a\in \Z$.  For any places $v$ above $p\equiv 1 \mod 12$, the character $\chi$ takes values $\pm 1$ and  the characteristic polynomial of $\sigma_{\l,4,-1}(\Frob_v)$ is of the form $X^2+ a  X+p^2$ for some $a\in \Z$.
Thus, the representation $\rho_{f} \mid _{H}\otimes \chi $ takes values in
$\Q _{\ell}(\sqrt{-2})$, as does the representation  $(\sigma_{\l,4,-1})^{\vee}$.
In the proof below, we will see that it suffices to compare  $c_p(f)\chi(v)$ with the trace of
$\sigma_{\lambda, 4, -1}( \Frob_v)$,  that is, the  trace of $(\sigma_{\lambda, 4, -1})^{\vee} (\mathrm{Fr}_v)$ for every place $v$ of $\Q(i)$ above the primes $p=5$ and 13.

Now we can use Faltings-Serre modularity criterion effectively.
We take the representation with coefficients in $\Q _2 (\sqrt{-2})$.
Let $\wp=(\sqrt{-2})$. We now consider
both representations modulo $\wp$ with
images in $\text{SL}_2(\F_2)$. For simplicity, use $\bar \rho$ to denote $\overline{\rho}_{\ell,f}|_{G_{\Q(\sqrt{-1})}}$.
The trace of
  $\bar \rho(\text{Fr}_{2+3i})=1$  so the
  image has order 3 elements. If the image  is $C_3$, then it gives rise
  to a cubic  extension of $\Q(i)$ which is unramified outside
  of $1+i$ and 3. Like Lemma 19 of \cite{long061}, we know such a
  cubic field is  the splitting field of
  $x^3-3x+1$. It is irreducible mod $2+i$, but the
characteristic polynomial
  of $\bar \rho(\text{Fr}_{1+2i})$ is $T^2+1$,
  hence it is of order 1 or 2 which leads to a contradiction.
So $\ker \bar \rho$ corresponds an $S_3$-extension of
$\Q(i)$ which can be identified as  the splitting field of $x^3-2$
over
  $\Q(i)$.

Now we are going to consider all possible deviation groups
$\widetilde{G}$ measuring the  difference between $(\sigma_{\l,4,-1} )^{\vee}$ and
$(\rho_{f}\mid _{G_{\Q(i)}})\otimes \chi$ if they are not isomorphic up to
semisimplification.

Let $F'$ be the splitting field of $x^3-2$ over $\Q$. Similar to the
proof of Lemma 19 of \cite{long061}, we first look for
$S_4$-extensions $L'$ of $\Q$ containing $F'$ but not $\Q(i)$ and
unramified outside of 2 and 3. By using the fact that their cubic
resolvent is $x^3-2$ we conclude that such $S_4$ extensions are the
splitting fields of irreducible polynomials of the form
$x^4+ux^2+vx-u^2/12$ where $u,v\in \Z[1/6]$ such that
$8u^3/27+v^2=\pm 2$. The possible polynomials giving non-isomorphic fields are
$$x^4+6x^2+8x-3, \quad x^4-18x^2+40x-27.$$ $\text{Fr}_5$ has order 4 in
these fields. By comparing the traces of $\Frob_{1\pm 2i}$ (or $\text{Fr}_{1\pm 2i}$) for both
representations, we eliminate the possibility that $\widetilde{G}$
contains $S_4$ as a subgroup.

The remaining possibility is  $\widetilde{G}=S_3\times \Z /2$. Among
all the quadratic extensions $\Q(i,\sqrt{d})$ of $\Q(i)$
  unramified at $1+i$ and 3: $d=i,1+i,1-i,3+3i,3-3i$, we have that $\text{Fr}_v$ for  $v=3+2i, 3+2i, 3-2i, 3+2i,3-2i$ are inert in these fields
  respectively. Meanwhile these Frobenius elements map under $\bar{\rho}$ to an element of order
  3. Such an element in the conjugacy class of $\text{Fr}_v$ has order 6 as it has order 3 in the $S_3$ component and order 2 in the $\Z /2$ component. Since the characteristic polynomials of
  $\Frob_{3\pm 2i}$ (or $\text{Fr}_{3\pm 2i}$) under both representations agree, we know $\widetilde{G}$ cannot be $S_3\times \Z /2$ either, because if the two representations were different,
they would  have different traces at elements of order 4 or higher in the deviation group, see \cite{lly05}.
\end{proof}

\begin{cor}
  $$L(s,\rho_{\ell,4}^{\vee})=L(s,(\sigma_{\lambda, 4, -1})^{\vee})=L(s,\rho_{f} \mid _{G_{\Q(i)}}\otimes \chi )$$ where $\rho_{f} \mid _{G_{\Q(i)}}\otimes \chi$ is a $\text{GL}_2(\mathbb A_\Q)\times \text{GL}_2(\mathbb A_\Q)$  and hence a $\text{GL}_4(\mathbb A_\Q)$ automorphic form by a result of D. Ramakrishnan \cite{Ramakrishnan00}.
\end{cor}

\section{ Atkin-Swinnerton-Dyer congruences}
\label{S:aswd}
\subsection{}
\label{SS:aswd1}
Let $V_{p, a}$ be the 4-dimensional $F$-crystal ($\Q_p$ vector space with
Frobenius action) associated with $S_3(\Gamma)$ in the model of the curve
$X(\Gamma)$ with Hauptmodul $r_a$. Note that the only difference among the
$V_{p, a}$ is the action of the Frobenius $F$. We need to consider
both $a = 2, 4$: for $a = 4$ it is easier to describe the modularity of the $\ell$-adic counterpart
 $\rho_{\ell, 4}$ (see Theorem \ref{thm:mod-main}) and the factorization of $H_{p,4}(X)$, but to prove 3-term ASD
congruences in the simplest form for the expansion coefficients of the cuspforms $h_1, h_2$ relative to the $q$-parameter via the results of
Scholl's paper \cite{Sch85i}, \cite{Sch85ii}, we need $a=2$. %that
%there is a $\Q$-rational point on the model of the curve
%$X(\Gamma)$ lying over the point $\tau = i \infty$, and this is false
%in the model given by the Hauptmodul $r_4$, but true in the model given by
%$r_2$ (see Section \ref{SS:corr4}).
We could derive ASD congruences using the $F$-crystal $V_{p, 4}$ but these would
be expressed in a parameter $\gamma q$ for some algebraic number
$\gamma$. Therefore we only consider the $F$-crystal $V_p = V_{p, 2}$.

\subsection{}
\label{SS:aswd2}
The method is the following. We use the operators $B_s$ (resp. $B_s^*$)
$s = -1, \pm 3$ to decompose both $V_{p}$ (resp. $W_{\ell}:=W_{\ell,2}$)
into eigenspaces $V^{\pm}_{p, s}$ (resp. $W_{\lambda,s}^{\pm}$). By the way $B_s$ are defined, we know that
$B_s$ is defined over the field  $L_{-1}=\Q (i)\cdot \Q(\sqrt[3]{2})$,  $L_{-3}=\Q (\sqrt{-3})$ and $L_3=\Q (\sqrt{3}) \cdot \Q(\sqrt[3]{2})$ for
$s = -1, -3, 3$ respectively (see \S 3.5). We show that  for each prime $p \ge 5$
there is an $s$ such that the Frobenius $F$ acts on the $2$-dimensional
eigenspaces $V^{\pm}_{p,s}$ which corresponds to an $\Frob_p$ in $G_{L_s}$ acting on   $W_{\lambda,s}^{\pm}$. % (this is done by fixing a embedding of $\sqrt[3]{2}$ in $\Q_p$ when $p=2\mod 3$.
For $p \equiv \ 1$ mod $3$ we use $s = -3$; for
$p \equiv \ 5$ mod $12$ we use $s = -1$; for $p \equiv \ 11$ mod $12$ we use $s = 3$.
Following the approach of Scholl \cite{Sch85ii}, there is a comparison theorem between Frobenius action on both $\ell$-adic and $p$-adic spaces that implies
\begin{equation}\label{eq:4}
 \mathrm{Char}(F, V^{\pm}_{p,s}, X)  =  \mathrm{Char}(\mathrm{Frob}_{p},W_{\lambda,s}^{\pm} , X),
\end{equation} where  the right hand side can be computed from $f$, $\psi$ and $\chi$. (Also the Frobenius $\Frob_p$ needs further justification to ensure it commutes with $B_s^*$.)
In particular, for primes $p \equiv\ 2 \mod 3$,   the right hand side
can be computed from $f$ and $\chi$ only by using Lemma \ref{cor:2}.
 We find the ASD basis as linear combinations
$h_1 \pm \alpha h_2$ which span the $1$-dimensional spaces
\[
 V_{p}^{\pm} \cap (S_{3}(\Gamma)\otimes\overline {\Q}_p).
\]
By the Cayley-Hamilton theorem, $\mathrm{Char}(F, V^{\pm}_p, F)(g)  = 0$
for any $g \in  V^{\pm}_p$. Writing out the Frobenius action in
the local coordinate $r_2$ applied to the eigenfunctions
$h_1 \pm \alpha h_2$, this gives the desired three term ASD congruences.

\subsection{$p \equiv\ 1$ mod $3$}
\label{SS:aswd3}
\begin{prop}\label{prop:p=1 mod 3}For every prime  $p \equiv 1 \mod 3$,
  $h_1$, $h_2$  form a basis of $S_3(\G)$ for the three-term ASD
congruences at $p$ given by the characteristic polynomial
$X^2-\tau_i(a_p)X+ \tau_i(b_p)p^2 \in \Q_p[X]$ where $a_p, b_p\in
\Q(\sqrt{-3})$, $b_p$ a 6th root of unity, and $\tau_1,
\tau_2$ are two different embeddings of $\Q(\sqrt{-3})$ into $\Q_p$.
Moreover, $a_p$ differs from the $p$th coefficients of $f$ by at
most a twelfth root of unity that can be determined by $\psi$ and $\chi$ defined as before.
\end{prop}
\begin{proof}Let $p\equiv 1 \mod 3$ be a prime.
 The cuspforms $h_1$ and $h_2$ are distinct eigenvectors  the $B_{-3}$-operator.
On $W_\ell$, the corresponding operator $B_{-3}^*$  is defined over $\Q(\sqrt{-3})$, hence commutes with $\Frob_p$.  Consequently,
$$H_{p,2}(X)=(X^2-a_pX+b_p p^2)(X^2-\bar a_p X +\bar b_p p^2)$$ for some $a_p,b_p\in \Q(\sqrt{-3})$ and the bar denotes complex conjugation. The value of $b_p$ can be determined by Theorem \ref{thm:mod-main-1} and Theorem \ref{thm:mod-main}.
 The claim follows from \eqref{eq:4} and the discussion in the beginning of this section.

\end{proof}

\subsection{$p \equiv \ 2$ mod $3$}
\label{SS:aswd4}
In this case $X^3-2$ has exactly one root in $\F_p$ which
gives rise to a unique embedding   of $\Q(\sqrt[3]{2})$ in
$\Q_p$. In the sequel, we regard $\sqrt[3]{2}$ as an element in
$\Q_p$.

\begin{prop}\label{prop:p=5}
  When $p\equiv5 \mod 12$, let $\tau$ be an embedding of $\Q(i,\sqrt{2})$
  to $\Q_p(\sqrt{2})$.
 The functions
  $h_1\pm \frac{\tau(2^{1/2})}{2^{1/3}}h_2$ form a basis for the three-term ASD
congruences at $p$ given by the characteristic polynomial
$X^2\pm a_p\tau(\sqrt{-2})X-p^2\in \Q_p(\sqrt{2})[X]$ where $a_p\in
\Z$ and $a_p\cdot \sqrt{-2}$ differs from the $p$th coefficients of $f$ by
a fourth root of unity.
\end{prop}
\begin{proof}Let $p\equiv 5 \mod 12$ be a prime.
The operator $B_{-1}$ is defined over the field $L_{-1}=\Q (\sqrt{-1})\cdot \Q(\sqrt[3]{2}).$ In  the ring of integers of $L_{-1}$ there is a unique place above $p$ with relative degree 1. By abusing notation, we denote this place by $p$ again. Under this assumption $\mathrm{Frob}_p$ commutes with $B_{-1}$. Thus $\mathrm{Frob}_p$ (as an element of $G_{L_{-1}}$) acts on the
eigenspace $W^{\pm}_{\lambda,-1}$ and \eqref{eq:4} holds. By Lemma \ref{cor:2}, $$\text{Char}(F, V_{\l,-1}^{\pm},X)= \text{Char}(\Frob_p, W_{\l,2, -1}^{\pm},X)= \text{Char}(\Frob_p, W_{\l,4, -1}^{\pm},X).$$ As a consequence of $\rho_{\ell,4}$ satisfying the quaternion multiplication, we know that $$\text{Char}(\Frob_p, W_{\l,4, -1}^{\pm},X)=X^2\pm a_p\sqrt{-2}X-p^2$$ for some $a_p\in \Z$. By Theorem \ref{thm:mod-main}, $a_p\sqrt{-2}$ is different from the $p$th coefficient of $f$ by  a 4th fourth root of unity  as $\chi$ takes value $\pm i$ for any place $v$ above $p\equiv5 \mod 12$ in $\Z[i]$. The eigenfunctions of $B_{-1}$ are $h_1\pm
\frac{2^{1/2}}{2^{1/3}}h_2$. For any embedding $\tau$ of
$\Q(i,\sqrt{2})$
  to $\Q_p(\sqrt{2})$, $h_1\pm
 \frac{\tau(2^{1/2})}{2^{1/3}}h_2$ is a formal power series in
 $\Q_p(\sqrt{2})[[q]].$  It  satisfies the three term  ASD
congruences at $p$ given by $X^2\pm \tau(a_p) X-p^2$ as claimed.
\end{proof}

\begin{prop}\label{prop:p=11}
  When $p\equiv 11 \mod 12$,   let $\tau$ be an embedding of $\Q(\sqrt{3},\sqrt{-2})$
  to $\Q_p(\sqrt{-2})$.
 The functions
  $h_1\pm \frac{\tau((-2)^{1/2})}{2^{1/3}}h_2$ form a basis for the three-term ASD
congruences at $p$ given by the characteristic polynomial
$X^2\pm a_p\tau(\sqrt{-6})X-p^2\in \Q_p[T]$ where $a_p\in \Z$  and $a_p\tau(\sqrt{-6})$
differs from the $p$th coefficients of $f$ by at most a most a $\pm$ sign.
\end{prop}
\begin{proof}
 The reasoning is similar to the previous case but with the operator $B_3$, defined
over $\Q (\sqrt{3})\cdot\Q(\sqrt[3]{2})$ with $B_3 ^2 = -24$.
\end{proof}

\section{tables}
\label{S:tab}

In tables \ref{T:moddata}, \ref{T:moddata2} we display the factor $g_{p,a}(X)$ such that the
characteristic polynomial of $\Frob_p$ is $ H_{p,a}(X)= g_{p,a}(X)
\overline{g_{p,a}(X)}$. When $a=2$, we write this in the form
\[
g_{p,2}(X) = X^2 - \zeta c_p(f)X + (-4/p)\zeta ^2p^2
\]
for a twelfth root of unity $\zeta$.

{\small
\begin{table}[h]
\begin{tabular}{|c|c|c|c|}
\hline
$p$ &  $g_{p,2}(X)$ &$ c_p(f)$ & $\zeta$\\
\hline
5 & $X^2 + 6\sqrt{-2} X - 5^2$ & $6\sqrt{2}$ & $i$ \\
\hline 7 & $X^2 + \sqrt{-3}\left (\tfrac{-1-\sqrt{-3}}{2}\right ) X
         - \left (\tfrac{-1+\sqrt{-3}}{2}\right )7^2$
           & $-\sqrt{-3}$ & $\omega ^4$ \\
\hline
11 & $X^2 + 6\sqrt{-6} X - 11^2$ & $-6\sqrt{-6}$ & $1$ \\
\hline 13 & $X^2 -13 \left (\tfrac{1+\sqrt{-3}}{2}\right ) X
         + \left (\tfrac{-1+\sqrt{-3}}{2}\right )13^2$
           & $13$ & $\omega $ \\
\hline
17 & $X^2 + 6\sqrt{-2} X - 17^2$ & $-6\sqrt{2}$ & $i$ \\
\hline 19 & $X^2 + 11\sqrt{-3}\left (\tfrac{-1+\sqrt{-3}}{2}\right )X
         - \left (\tfrac{-1-\sqrt{-3}}{2}\right )19^2$
           & $-11\sqrt{-3}$ & $\omega ^2$ \\
 \hline
23 & $X^2 -18\sqrt{-6} X - 23^2$ & $18\sqrt{-6}$ & $1$ \\
\hline 29 & $X^2 +24\sqrt{-2} X  -29^2$
           & $-24\sqrt{2}$ & $i$ \\
\hline 31 & $X^2 -24\sqrt{-3} X
         -31^2$
           & $24\sqrt{-3}$ & $\omega^6$ \\

 \hline 37 & $X^2 -35 \left (\tfrac{-1-\sqrt{-3}}{2}\right ) X
         + \left (\tfrac{-1+\sqrt{-3}}{2}\right )37^2$
           & $35$ & $\omega ^4 $ \\
\hline
41 & $X^2 -41^2$ &0& $i$ \\
\hline 43 & $X^2 + 24\sqrt{-3}X
         - 43^2$
           & $-24\sqrt{-3}$ & $\omega ^6$ \\
\hline
47 & $X^2 -6\sqrt{-6} X - 47^2$ & $6\sqrt{-6}$ & $1$ \\
\hline 53 & $X^2 -36\sqrt{-2} X -53^2$
           & $36\sqrt{2}$ & $i$ \\
\hline
59 & $X^2 -30\sqrt{-6} X - 59^2$ & $30\sqrt{-6}$ & $1$ \\
\hline
\end{tabular}
\caption{ Factorization of $ H _{p, 2}= g_{p,2}(X)\overline{g_{p,2}(X)} $, coefficients of $f$;
$\omega = \mathrm{exp}(2 \pi i /6)$, $i = \sqrt{-1}$.} \label{T:moddata}
\end{table}}

{\small
\begin{table}[h]
\begin{tabular}{|c|c|c|}
\hline
$p$ &  $g_{p,4}(X)$ &$ c_p(f)$ \\
\hline
5 & $X^2 + 6\sqrt{-2} X - 5^2$ & $6\sqrt{2}$ \\
\hline 7 & $X^2 + \sqrt{-3}X
         - 7^2$
           & $-\sqrt{-3}$  \\
\hline
11 & $X^2 + 6\sqrt{-6} X - 11^2$ & $-6\sqrt{-6}$  \\
\hline 13 & $X^2 +13 X
         + 13^2$
           & $13$ \\
\hline
17 & $X^2 + 6\sqrt{-2} X - 17^2$ & $-6\sqrt{2}$\\
\hline 19 & $X^2 + 11\sqrt{-3}
X
         - 19^2$
           & $-11\sqrt{-3}$  \\
 \hline

23 & $X^2 -18\sqrt{-6} X - 23^2$ & $18\sqrt{-6}$  \\
\hline 29 & $X^2 +24\sqrt{-2} X  -29^2$
           & $-24\sqrt{2}$  \\
\hline 31 & $X^2 -24\sqrt{-3} X
         -31^2$
           & $24\sqrt{-3}$  \\
 \hline
37 & $X^2 -35  X
         + 37^2$
           & $35$ \\
\hline
41 & $X^2 -41^2$ &0 \\
\hline 43 & $X^2 + 24\sqrt{-3}X
         - 43^2$
           & $-24\sqrt{-3}$ \\
\hline
47 & $X^2 -6\sqrt{-6} X - 47^2$ & $6\sqrt{-6}$  \\
\hline 53 & $X^2 -36\sqrt{-2} X -53^2$
           & $36\sqrt{2}$  \\
\hline
59 & $X^2 -30\sqrt{-6} X - 59^2$ & $30\sqrt{-6}$  \\
%\hline 61 & $X^2 -83  X
%         + 61^2$
%           & $83$  \\
%\hline 67 & $X^2 + 13\sqrt{-3} X -  67^2$
%           & $-13\sqrt{-3}$  \\
%\hline
%71 & $X^2 +12\sqrt{-6} X - 71^2$ & $-12\sqrt{-6}$  \\
%\hline 73 & $X^2 -71  X
%         + 73^2$
%           & $-71$  \\
%\hline 79 & $X^2 + 49\sqrt{-3} X -  79^2$
%           & $-49\sqrt{-3}$  \\
%\hline
%83 & $X^2 +60\sqrt{-6} X - 83^2$ & $-60\sqrt{-6}$ \\
%\hline 89 & $X^2 +102\sqrt{-2} X -89^2$
%           & $-102\sqrt{2}$  \\
%\hline 97 & $X^2 +25  X
%         + 97^2$
%           & $25$  \\
%\hline 101 & $X^2 +72\sqrt{-2} X -101^2$
%           & $-72\sqrt{2}$  \\
%\hline 103 & $X^2 +49\sqrt{-3} X-  103^2$
%            & $-49\sqrt{-3}$  \\
%\hline
%107 & $X^2 +30\sqrt{-6} X - 107^2$ & $-30\sqrt{-6}$ \\
%\hline 109 & $X^2 -146  X + 109^2$
%           & $-146$  \\
%\hline 157 & $X^2 +190  X + 157^2$
%           & $-190$ & $\omega ^6 $ \\
%\hline 229 & $X^2 -170  X + 229^2$
%           & $-170$ & $\omega ^3 $ \\
%\hline 277 & $X^2 +94  X + 277^2$
%           & $94$ & $\omega ^3 $ \\
\hline
\end{tabular}
\caption{Factorization of $ H _{p,4}= g_{p,4}(X)\overline{g_{p,4}(X)} $, coefficients of $f$.}
\label{T:moddata2}
\end{table}}

%\bibliographystyle{amsalpha}
%\bibliography{longbibl}

\newcommand{\etalchar}[1]{$^{#1}$}
\providecommand{\bysame}{\leavevmode\hbox
to3em{\hrulefill}\thinspace}
\providecommand{\MR}{\relax\ifhmode\unskip\space\fi MR }
\providecommand{\MRhref}[2]{%
  \href{http://www.ams.org/mathscinet-getitem?mr=#1}{#2}
} \providecommand{\href}[2]{#2}

\end{document}